\documentclass[12pt]{article}
\usepackage{e-jc_noname}  
\usepackage{amsthm,amsmath,amssymb}
\usepackage{graphicx}
\usepackage[colorlinks=true,citecolor=black,linkcolor=black,urlcolor=blue]{hyperref}

\theoremstyle{plain}
\newtheorem{theorem}{Theorem}

\theoremstyle{definition}

\newtheorem{conjecture}[theorem]{Conjecture}

\theoremstyle{remark}

\title{Burling graphs, chromatic number, and orthogonal tree-decompositions} 

\author{Stefan Felsner \\
\small Institut f\"ur Mathematik\\[-0.8ex]
\small Technische Universit\"at Berlin\\[-0.8ex] 
\small Berlin, Germany\\
\small\tt felsner@math.tu-berlin.de\\
\and
Gwena\"{e}l Joret\thanks{Supported by an ARC grant from the Wallonia-Brussels Federation of Belgium.} \\
\small Computer Science Department\\[-0.8ex]
\small Universit\'e Libre de Bruxelles\\[-0.8ex]
\small Brussels, Belgium\\
\small\tt gjoret@ulb.ac.be \\
\and 
Piotr Micek\thanks{Supported by a Polish National Science Center grant (SONATA BIS 5; UMO-2015/18/E/ST6/00299). } \\
\small Theoretical Computer Science Department\\[-0.8ex]
\small Jagiellonian University\\[-0.8ex]
\small Krak\'ow, Poland\\[-0.8ex]
\small Institute of Mathematics\\[-0.8ex]
\small Freie Universit\"at Berlin\\[-0.8ex]
\small Berlin, Germany \\
\small\tt piotr.micek@tcs.uj.edu.pl \\
\and 
William T. Trotter \\
\small School of Mathematics\\[-0.8ex]
\small Georgia Institute of Technology\\[-0.8ex]
\small Atlanta, Georgia, U.S.A.\\
\small\tt trotter@math.gatech.edu \\
\and 
Veit Wiechert\thanks{Supported by the Deutsche Forschungsgemeinschaft within the research training group `Methods for Discrete Structures' (GRK 1408).} \\
\small Institut f\"ur Mathematik\\[-0.8ex]
\small Technische Universit\"at Berlin\\[-0.8ex] 
\small Berlin, Germany\\
\small\tt wiechert@math.tu-berlin.de
}

\newcommand{\calC}{\mathcal{C}}

\newcommand{\calS}{\mathcal{S}}

\let\leq\leqslant
\let\geq\geqslant

\let\epsilon\varepsilon

\begin{document}

\maketitle

\begin{abstract}
	A classic result of Asplund and Gr\"unbaum states that intersection graphs of axis-aligned rectangles in the plane are $\chi$-bounded. 	
	This theorem can be equivalently stated in terms of path-decompositions as follows:  There exists a function $f:\mathbb{N}\to\mathbb{N}$ such that every graph that has two path-decompositions such that each bag of the first decomposition intersects each bag of the second in at most $k$ vertices has chromatic number at most $f(k)$.  
	Recently, Dujmovi\'c, Joret, Morin,  Norin, and Wood asked whether this remains true more generally for two tree-decompositions. 
	In this note we provide a negative answer: There are graphs with arbitrarily large chromatic number for which one can find two tree-decompositions such that each bag of the first decomposition intersects each bag of the second in at most two vertices. 
	Furthermore, this remains true even if one of the two decompositions is restricted to be a path-decomposition. 
This is shown using a construction of triangle-free graphs with unbounded chromatic number due to Burling, which we believe should be more widely known. 
\end{abstract}

\section{Burling graphs}

For each $k \geq 1$, we define the {\em Burling graph} $G_k$ and a collection $\mathcal{S}(G_k)$ of stable sets of $G_k$ by induction on $k$ as follows. 
First, let $G_1$ be the graph consisting of a single vertex and let $\mathcal{S}(G_1)$ contain just the single vertex stable set of $G_1$.  
Next, suppose $k \geq 2$ for the inductive case.  
First, take a copy $H$ of $G_{k-1}$, which we think of as the `master' copy. 
For each stable set $S\in \calS(H)$, let $H_S$ denote a new copy of $G_{k-1}$. 
Furthermore, for each stable set $X\in \calS(H_S)$, introduce a new vertex $v_{S, X}$ adjacent to all vertices in $X$ but no others.  
Let us denote by $H'_S$ the graph obtained from $H_S$ resulting from these vertex additions.  
The graph $G_k$ is then defined as the union of $H$ and $H'_S$ over all $S\in \calS(H)$. 
Its collection  $\mathcal{S}(G_k)$ consists of two sets for each $S\in\calS(H)$ and $X\in\calS(H_S)$, namely: $S\cup X$ and $S\cup \{v_{S,X}\}$. 
Observe that $S\cup X$ and $S\cup \{v_{S,X}\}$ are both stable sets of $G_k$. 

Burling defined the family $\{G_k\}$ in his PhD Thesis~\cite{Bur65} in 1965 and proved that these graphs have unbounded chromatic number. 
However, this construction went mostly unnoticed until it was rediscovered in~\cite{segment}.  
(One exception is a set of unpublished lecture notes of Gy\'arf\'as~\cite{G03} from 2003, which has a section devoted to Burling graphs.) 

\begin{theorem}[\cite{Bur65}]
For every $k \geq 1$, the Burling graph $G_k$ is triangle free and has chromatic number at least $k$.
\end{theorem}
\begin{proof}
The fact that $G_k$ is triangle free follows directly by observing that, when creating a vertex $v_{S,X}$ in the definition of $G_k$, its neighborhood is a stable set.  
To show that $\chi(G_k) \geq k$, we prove the following stronger statement by induction on $k$: 
For every proper coloring $\phi$ of $G_k$, there exists a stable set $S \in \calS(G_k)$ such that $\phi$ uses at least $k$ colors for vertices in $S$. 
This is obviously true for $k=1$, so let us assume $k \geq 2$ and consider the inductive case. 
Let $\phi$ be a proper coloring of $G_k$. 
In what follows, the notations $H$, $H_S$, and $H'_S$ refer to the graphs used in the definition of $G_k$. 
By induction, there is a stable set $S\in \calS(H)$ such that $\phi$ uses at least $k-1$ colors on $S$. 
Similarly, there is a stable set $X\in \calS(H_S)$ such that $\phi$ uses at least $k-1$ colors on $X$.  
If $\phi$ uses at least $k$ colors on $S\cup X$, we are done since $S\cup X \in \calS(G_k)$. 
If not, then $\phi$ uses exactly the same set $C$ of $k-1$ colors on $S$ and on $X$. 
This implies that the vertex $v_{S,X}$ is colored with a color not in $C$, and hence $\phi$ uses $k$ colors on the stable set $S\cup \{v_{S,X}\} \in \calS(G_k)$. 
\end{proof}

Mycielski~\cite{M55}, and Erd\H{o}s and Hajnal~\cite{EH68} each described easy constructions of triangle-free graphs with unbounded chromatic number that are classics nowadays.  
We believe that Burling graphs should be more widely known, for their definition is simple and yet they exhibit some unique properties. 
In particular, Burling graphs admit various geometric representations that are not known to exist for any other family of triangle-free graphs with unbounded chromatic number, which we briefly survey now. 
 
First, recall that a class of graphs $\calC$ is \emph{$\chi$-bounded} if there is a function $f$ such that $\chi(G)\leq f(\omega(G))$ for all $G\in \calC$, where $\omega(G)$ denotes the maximum size of a clique in $G$. 

Burling~\cite{Bur65} showed that each $G_k$ can be obtained as the intersection graph of axis-aligned boxes in $\mathbb{R}^3$. 
Hence, this implies that intersection graphs of axis-aligned boxes in $\mathbb{R}^3$ are not $\chi$-bounded. 
This is in contrast with the result of Asplund and Gr\"unbaum~\cite{AG60} that $\chi(G) \in O(\omega^2(G))$ for intersection graphs $G$ of axis-aligned rectangles. 
(We remark that Reed and Allwright~\cite{RA08} (see also~\cite{MM11}) described another interesting construction of axis-aligned boxes in $\mathbb{R}^3$ whose intersection graph has high chromatic number, with the extra property that the interiors of the boxes are pairwise disjoint, implying that the clique number is at most $4$.)   

In the 1970s, Erd\H{o}s asked whether intersection graphs of line segments in the plane are $\chi$-bounded. 
A negative answer was provided by Pawlik, Kozik, Krawczyk, Laso\'n, Micek, Trotter, and Walczak~\cite{segment}: The authors represented the Burling graphs as intersection graphs of segments in the plane. 
This result also disproves the conjecture of Scott~\cite{Scott97} that, for every graph $H$, the class of graphs excluding every subdivision of $H$ as an induced subgraph is $\chi$-bounded.
Indeed, segment intersection graphs---and thus in particular Burling graphs---do not contain any subdivision of $H$ as an induced subgraph when $H$ is the $1$-subdivision of a non-planar graph. 
Later on, Chalopin, Esperet, Li, and Ossona de Mendez~\cite{Chalopin16} showed that Burling graphs in fact even exclude all subdivisions of $H$ as an induced subgraph when $H$ is the $1$-subdivision of $K_4$.

\section{Application to orthogonal tree-decompositions}

A {\em tree-decomposition} of a graph $G$ is a pair $(T,\{B_t\}_{t\in V(T)})$ where $T$ is a tree and the sets $B_t$ ($t \in V(T)$) are subsets of $V(G)$ called {\em bags} satisfying the following properties: 
\begin{enumerate}
\item for each edge $uv \in E(G)$ there is a bag containing both $u$ and $v$, and
\item for each vertex $v\in V(G)$, the set of vertices $t\in V(T)$  with $v \in B_t$ induces a non-empty subtree of $T$. 
\end{enumerate}
The {\em width} of the tree-decomposition is the maximum size of a bag minus $1$. 
The {\em tree-width} of $G$ is the minimum width of tree-decompositions of $G$. 
Path-decompositions and path-width are defined analogously, with the extra requirement that the tree $T$ be a path. 
We refer the reader to Diestel~\cite{D10} for background on tree-decompositions. 

The following generalization of tree-decompositions was recently introduced by Stavropoulos~\cite{K15, K16} and investigated by Dujmovi\'c, Joret, Morin,  Norin, and Wood~\cite{DJMNW17}.  
Suppose that $(T^1,\{B^1_t\}_{t\in V(T^1)}), \dots, (T^k,\{B^k_t\}_{t\in V(T^k)})$ are $k$ tree-decompositions of a graph $G$. 
Let then the {\em $k$-width} of these decompositions be the maximum of $|B^1_{t_1} \cap \cdots \cap B^k_{t_k}|$ over all $(t_1, \dots, t_k) \in V(T_1) \times \cdots \times V(T_k)$. 
The {\em $k$-tree-width} of $G$, also called {\em $k$-medianwidth} of $G$ in~\cite{K15, K16}, is the minimum $k$-width of all $k$-tuples of tree-decompositions of $G$.  
Replacing trees with paths, we obtain the corresponding notion of {\em $k$-path-width} of $G$, also known as {\em $k$-latticewidth}~\cite{K16}. 
Intuitively, to show that the $k$-tree-width or $k$-path-width of $G$ is small, we want to choose a $k$-tuple of tree/path-decompositions of $G$ that are as `orthogonal' as possible: 
For instance, to see that a grid has bounded $2$-path-width, one can take a `horizontal' path-decomposition where bags are unions of two consecutive columns, and a `vertical' one where bags are unions of two consecutive rows. 
 
The $k$-tree-width of $G$ for $k=1, 2, 3, \dots$ forms a non-increasing sequence of numbers that converges to the clique number $\omega(G)$ of $G$,  and the same is true for the $k$-path-width of $G$~\cite{K15, K16}. 
Thus these numbers can be seen as interpolating between the tree-width / path-width of $G$  (plus one) and its clique number. 

Some graph classes of interest already have bounded $2$-tree-width. 
For instance, planar graphs, and more generally graphs excluding a fixed graph $H$ as minor~\cite{DJMNW17}. 
In fact, for planar graphs and some of their generalizations, one can even require one of the two tree-decompositions to be a path-decomposition such that each vertex appears in at most two bags, see~\cite{DJMNW17} and the references therein. 
Note however that graphs with bounded $2$-tree-width are not necessarily sparse: All bipartite graphs have $2$-tree-width (and even $2$-path-width) at most $2$. 

The $k$-path-width of a graph $G$ can equivalently be defined as the minimum $q$ such that $G$ is a subgraph of an intersection graph $H$ of axis-aligned boxes in $\mathbb{R}^k$ with $\omega(H) \leq q$. 
(To see this, recall that axis-aligned boxes in $\mathbb{R}^k$ satisfy the Helly property.) 
In particular, $\chi(G)$ is bounded from above by a function of the $2$-path-width of $G$, since intersection graphs of axis-aligned rectangles in the plane are $\chi$-bounded~\cite{AG60}. 
This prompted the authors of~\cite{DJMNW17} to ask whether the same remains true for the $2$-tree-width of $G$. 
We show that this is not the case, even if one the two decompositions is restricted to be a path-decomposition. 

\begin{theorem}
For every $k\geq 1$, the Burling graph $G_k$ has a tree-decomposition $(T,\{B_t\}_{t\in V(T)})$ and a path-decomposition $(P,\{B_p\}_{p\in V(P)})$ such that $|B_t\cap B_p|\leq 2$ for every $t\in V(T)$ and every $p\in V(P)$.
\end{theorem}
\begin{proof}
The proof is by induction on $k$. 
To facilitate the induction, we will prove  that the tree-decomposition and the path-decomposition can be chosen such that 
\begin{enumerate}
    \item $|B_t\cap B_p|\leq 2$ for every $t\in V(T)$ and every $p\in V(P)$; \label{prop1}
    \item for every $S \in \calS(G_k)$, there exists $t\in V(T)$ such that $B_t = S$, and \label{prop2}
    \item $|S \cap B_p| \leq 1$ for every $S \in \calS(G_k)$ and every $p\in V(P)$. \label{prop3}
\end{enumerate}

The claim is trivially true for $k=1$, so let us consider the inductive case $k \geq 2$. 
As before, the notations $H$, $H_S$, and $H'_S$ refer to the graphs used in the definition of $G_k$. 
Let $(T^H,\{B^H_t\}_{t\in V(T^H)})$ and $(P^H,\{B^H_p\}_{p\in V(P^H)})$ denote the tree-decomposition and path-decomposition of $H$ given by the induction hypothesis. 
Similarly, for each stable set $S\in \calS(H)$, let $(T^{H,S},\{B^{H,S}_t\}_{t\in V(T^{H,S})})$ and $(P^{H,S},\{B^{H,S}_p\}_{p\in V(P^{H,S})})$ denote the tree-decomposition and path-decomposition of $H_S$ obtained from induction. 
(As expected, we assume that $T^H$, $P^H$, and all the $T^{H,S}$s and $P^{H, S}$s are pairwise vertex disjoint.)

Define the tree $T$ as follows. 
Start with the union of $T^H$ and $T^{H,S}$ for all $S\in \calS(H)$. 
Then, for each $S\in \calS(H)$, add an edge linking a vertex $t \in V(T^H)$ such that $B^H_t = S$ (which exists by induction) to an arbitrary vertex in $V(T^{H, S})$.  
Finally, for each $S\in \calS(H)$ and $X \in \calS(H_S)$, let $t_{S,X}$ denote a vertex in $V(T^{H,S})$ such that $B^{H, S}_{t_{S, X}} = X$. 
Add two leaves $t^1_{S,X}$, $t^2_{S,X}$ adjacent to $t_{S,X}$. 

\begin{figure}[t]
\centering
\includegraphics{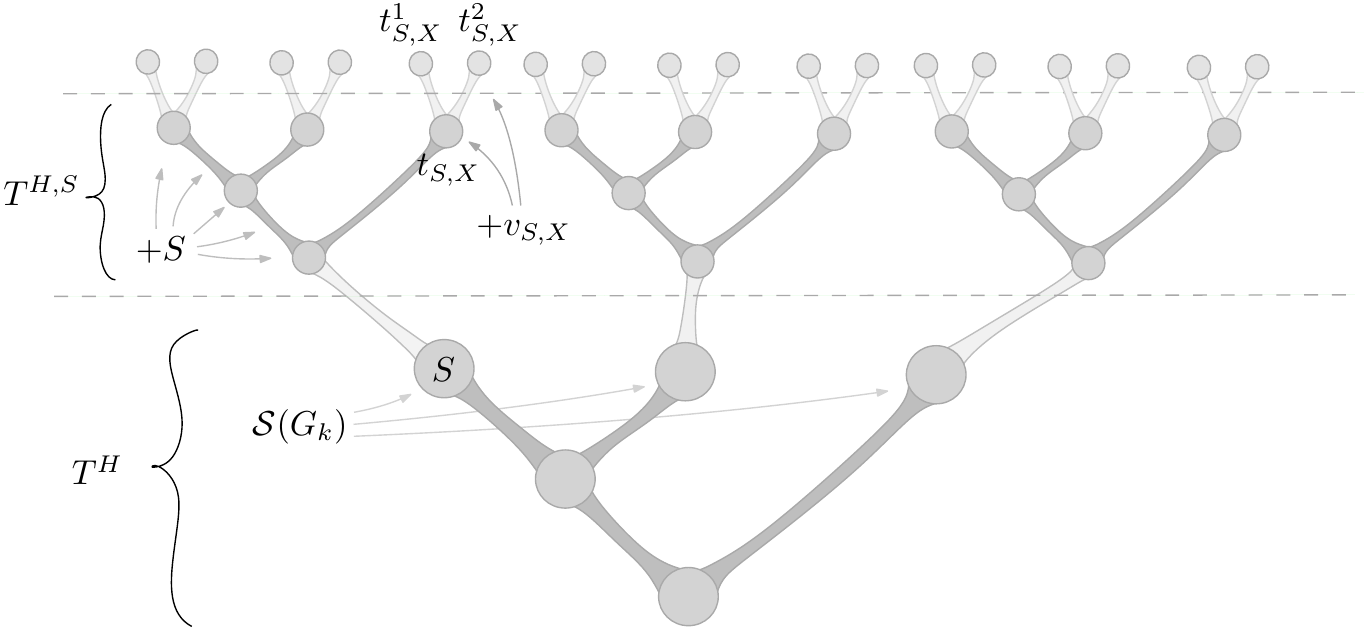}
\caption{Tree-decomposition of $G_k$.\label{fig:td}}
\end{figure}

The bags $B_t$ ($t\in V(T)$) of the tree-decomposition of $G_k$ are defined as follows (see Figure~\ref{fig:td} for an illustration) :  
$$
B_t := \left\{ 
\begin{array}{ll}
B^H_t & \textrm{ if } t\in V(T^H) \\[.5ex]
S \cup X \cup \{v_{S,X}\}   & \textrm{ if } t=t_{S,X} \textrm{ for some } S\in \calS(H) \textrm{ and } X \in \calS(H_S)\\[.5ex]
S \cup X   & \textrm{ if } t=t^1_{S,X} \textrm{ for some } S\in \calS(H) \textrm{ and } X \in \calS(H_S)\\[.5ex]
S \cup \{v_{S,X}\}   & \textrm{ if } t=t^2_{S,X} \textrm{ for some } S\in \calS(H) \textrm{ and } X \in \calS(H_S)\\[.5ex]
S \cup B^{H, S}_t  &\begin{array}{ll}\textrm{if }  t \in V(T^{H, S}) \textrm{ for some }  S\in \calS(H) \\ \quad \textrm{ and } t\neq t_{S, X} \textrm{ for all }X\in \calS(H_S) 
\end{array}
\end{array}
\right. 
$$

For each vertex $v\in V(G_k)$, the set of vertices $t\in V(T)$ such that $v\in B_t$ clearly induces a subtree of $T$. 
Moreover, the two endpoints of each new edge of the form $v_{S,X}x$ with $S\in \calS(H)$, $X\in \calS(H_S)$, and $x\in X$ lie in a common bag, namely $B_t$ with $t=t_{S,X}$.  
It follows that $(T,\{B_t\}_{t\in V(T)})$ is a tree-decomposition of $G_k$. 

We show that property~\ref{prop2} holds. 
Recall that  each set in $\calS(G_k)$ is either of the form $S \cup X$ or of the form $S\cup \{v_{S,X}\}$ for some  $S\in\calS(H)$ and $X\in\calS(H_S)$. 
In the former case, $S \cup X = B_t$ for $t=t^1_{S,X}$. 
In the latter case,  $S\cup \{v_{S,X}\} = B_t$ for $t=t^2_{S,X}$.  
Hence, \ref{prop2} is satisfied. 

Next, we define the path-decomposition of $G_k$. 
The path $P$ indexing the decomposition is defined simply by taking the union of the paths $P^H$ and  $P^{H,S}$ for all $S\in \calS(H)$, and connecting them in a path-like way (arbitrarily). 
The bags $B_p$ ($p\in V(P)$) are defined as follows (see Figure~\ref{fig:pd} for an illustration):  
$$
B_p := \left\{ 
\begin{array}{ll}
B^H_p & \textrm{ if } p\in V(P^H) \\[1ex]
B^{H, S}_p \cup  \{v_{S,X} \mid X\in\calS(H_S)\}  & \textrm{ if }  p \in V(P^{H, S}) \textrm{ for some }  S\in \calS(H)
\end{array}
\right. 
$$

\begin{figure}[t]
\centering
\includegraphics{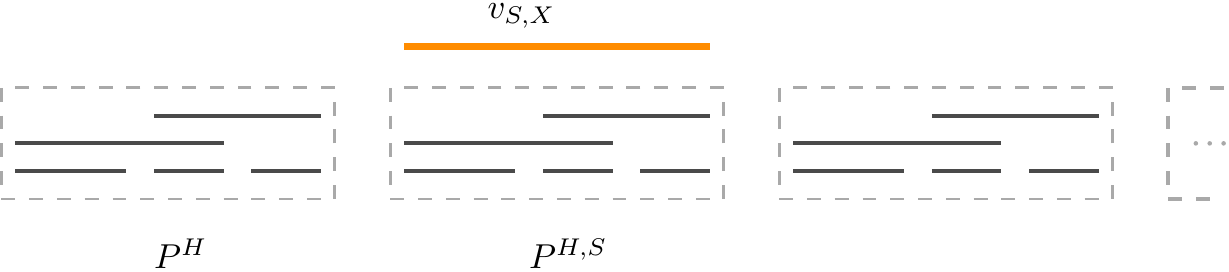}
\caption{Path-decomposition of $G_k$.\label{fig:pd}}
\end{figure}

Observe that $(P,\{B_p\}_{p\in V(P)})$ is a path-decomposition of $G_k$. 
Indeed, for each vertex $v\in V(G_k)$ the set of vertices $p\in V(P)$ such that $v\in B_p$ clearly induces a subpath of $P$. 
Moreover, the two endpoints of each new edge of the form $v_{S,X}x$ with $S\in \calS(H)$, $X\in \calS(H_S)$, and $x\in X$ lie in a common bag since $v_{S,X} \in B^{H,S}_p$ for every $p\in V(P^{H,S})$.

Let us prove that property~\ref{prop3} is satisfied.
Consider sets $S\in\calS(H)$ and $X\in\calS(H_S)$, and a vertex $p\in V(P)$. 
First suppose  $p \in V(P^H)$. 
Then $(S \cup X) \cap B_p = S \cap B^H_p$, and thus $|(S \cup X) \cap B_p| \leq 1$ holds by induction. 
Similarly, $(S \cup \{v_{S,X}\}) \cap B_p = S \cap B^H_p$  and again $|(S \cup  \{v_{S,X}\}) \cap B_p| \leq 1$ follows from induction. 
Next assume  $p \in V(P^{H, S})$ for some $S\in \calS(H)$. 
Then $(S \cup X) \cap B_p = X \cap B^{H,S}_p$, and thus $|(S \cup X) \cap B_p| \leq 1$ by induction. 
Also, $(S \cup \{v_{S,X}\}) \cap B^{H,S}_p =  \{v_{S,X}\}$  and hence $|(S \cup  \{v_{S,X}\}) \cap B_p| = 1$. 
It follows that property~\ref{prop3} holds.

It remains to show that our newly defined tree and path-decompositions together satisfy property~\ref{prop1}. 
Let thus  $t\in V(T)$ and  $p \in V(P)$. 
First, suppose that  $t\in V(T^H)$. 
If $p\in V(P^H)$, then $|B_t \cap B_p| \leq 2$ holds by induction. 
If $p \in V(P^{H, S})$ for some $S\in \calS(H)$, then $B_t$ and $B_p$ are disjoint.

Next, suppose that $t=t_{S,X}$ for some $S\in\calS(H)$ and $X\in\calS(H_S)$. 
Thus,  $B_t = S \cup X \cup \{v_{S,X}\}$. 
If $p\in V(P^H)$, then  $B_t \cap B_p = S \cap B^H_p$, and we know that this set has size at most $1$ by induction, since $H$ satisfies property~\ref{prop3}. 
If $p \in V(P^{H, S'})$ for some $S'\in \calS(H)$ distinct from $S$, then  $B_t$ and $B_p$ are disjoint.  
If $p \in V(P^{H, S})$, then  $B_t \cap B_p = (X \cap B^{H,S}_p) \cup \{v_{S,X}\}$. 
Since $|X \cap B^{H,S}_p| \leq 1$ holds by induction thanks to property~\ref{prop3}, we deduce that  $|B_t \cap B_p| \leq 2$.

The above observations also imply that $|B_t \cap B_p| \leq 2$ if $t=t^1_{S,X}$ or $t=t^2_{S,X}$  for some $S\in\calS(H)$ and $X\in\calS(H_S)$, since $B_t \subseteq S \cup X \cup \{v_{S,X}\}$ in these cases.   
Finally, suppose that $t \in V(T^{H, S})$ for some $S\in \calS(H)$ and $t\neq t_{S, X}$ for all $X\in \calS(H_S)$. 
Then $B_t = S \cup B^{H, S}_t$. 
If $p\in V(P^H)$, then  $B_t \cap B_p = S \cap B^H_p$, and (as in the above paragraph) that set has size at most $1$ by induction, since $H$ satisfies property~\ref{prop3}. 
If $p \in V(P^{H, S'})$ for some $S'\in \calS(H)$ distinct from $S$, then  $B_t$ and $B_p$ are disjoint.  
If $p \in V(P^{H, S})$, then  $|B_t \cap B_p| = |B^{H, S}_t \cap  B^{H, S}_p| \leq 2$ by induction. 

Hence,  $|B_t \cap B_p| \leq 2$ holds in all cases, and therefore property~\ref{prop1} is satisfied. 
\end{proof}

We conclude the paper with an open problem in the spirit of exploring how the Asplund-Gr\"unbaum result~\cite{AG60} could be extended. 
A {\em spaghetti tree-decomposition} of a graph $G$ is a tree-decomposition $(T,\{B_t\}_{t\in V(T)})$ of $G$ such that $T$ is rooted at some vertex $r\in V(T)$ and, orienting all edges of $T$ away from $r$, the subtree $T_v$ of $T$ induced by $\{t \in V(T): v \in B_t\}$ is a directed path for each vertex $v\in V(G)$. 

\begin{conjecture} 
There exists a function $f: \mathbb{N} \to \mathbb{N}$ such that $\chi(G) \leq f(k)$ for every $k \geq 1$ and every graph $G$ admitting a spaghetti tree-decomposition $(T,\{B_t\}_{t\in V(T)})$ and a path-decomposition $(P,\{B_p\}_{p\in V(P)})$ such that $|B_t \cap B_p| \leq k$ for every $t\in V(T)$ and $p\in V(P)$. 
\end{conjecture}

We remark that, for all we know, the above conjecture could even be true with two spaghetti tree-decompositions. 
Let us also mention that the class of graphs $G$ that admit a spaghetti tree-decomposition $(T,\{B_t\}_{t\in V(T)})$ such that $uv \in E(G)$ {\em if and only if} $T_u$ and $T_v$ intersect has been studied by Galvin~\cite{G75} (note that this is a subclass of chordal graphs). 

\subsection*{Acknowledgements} 

We thank Ross Kang for pointing out references~\cite{G03,RA08} and an anonymous referee for mentioning reference~\cite{G75}. 

\bibliographystyle{plain}
\bibliography{bibliography}

\begin{thebibliography}{10}

\bibitem{AG60}
E.~Asplund and B.~Gr\"unbaum.
\newblock On a coloring problem.
\newblock {\em Math. Scand.}, 8:181--188, 1960.

\bibitem{Bur65}
J.~P. Burling.
\newblock {\em On coloring problems of families of prototypes}.
\newblock PhD thesis, University of Colorado, Boulder, 1965.

\bibitem{Chalopin16}
J.~Chalopin, L.~Esperet, Z.~Li, and P.~Ossona~de Mendez.
\newblock Restricted frame graphs and a conjecture of {S}cott.
\newblock {\em Electron. J. Combin.}, 23(1):Paper 1.30, 21, 2016.
\newblock \arxiv{1406.0338}.

\bibitem{D10}
R.~Diestel.
\newblock {\em Graph theory}, volume 173 of {\em Graduate Texts in
  Mathematics}.
\newblock Springer, Heidelberg, fourth edition, 2010.

\bibitem{DJMNW17}
V.~Dujmovi{\'c}, G.~Joret, P.~Morin, S.~Norin, and D.~R. Wood.
\newblock Orthogonal tree decompositions of graphs.
\newblock {\em SIAM Journal on Discrete Mathematics}, to appear.
\newblock \arxiv{1701.05639}.

\bibitem{EH68}
P.~Erd\H{o}s and A.~Hajnal.
\newblock On chromatic number of infinite graphs.
\newblock In {\em Theory of {G}raphs ({P}roc. {C}olloq., {T}ihany, 1966)},
  pages 83--98. Academic Press, New York, 1968.

\bibitem{G75}
F.~Gavril.
\newblock A recognition algorithm for the intersection graphs of directed paths
  in directed trees.
\newblock {\em Discrete Math.}, 13(3):237--249, 1975.

\bibitem{G03}
A.~Gy\'arf\'as.
\newblock Combinatorics of intervals.
\newblock Unpublished lecture notes (2003).

\bibitem{MM11}
C.~Magnant and D.~M. Martin.
\newblock Coloring rectangular blocks in $3$-space.
\newblock {\em Discussiones Mathematicae Graph Theory}, 31:161--170, 2011.

\bibitem{M55}
J.~Mycielski.
\newblock Sur le coloriage des graphes.
\newblock {\em Colloq. Math.}, 3:161--162, 1955.

\bibitem{segment}
A.~Pawlik, J.~Kozik, T.~Krawczyk, M.~Laso\'n, P.~Micek, W.~T. Trotter, and
  B.~Walczak.
\newblock Triangle-free intersection graphs of line segments with large
  chromatic number.
\newblock {\em J. Combin. Theory Ser. B}, 105:6--10, 2014.
\newblock \arxiv{1209.1595}.

\bibitem{RA08}
B.~Reed and D.~Allwright.
\newblock Painting the office.
\newblock {\em MICS Journal}, 1:1--8, 2008.
\newblock Available
  \href{http://www.fields.utoronto.ca/journalarchive/mics/5-4.pdf}{online}.

\bibitem{Scott97}
A.~D. Scott.
\newblock Induced trees in graphs of large chromatic number.
\newblock {\em J. Graph Theory}, 24(4):297--311, 1997.

\bibitem{K16}
K.~Stavropoulos.
\newblock Cops, robber and medianwidth parameters.
\newblock \arxiv{1603.06871}.

\bibitem{K15}
K.~Stavropoulos.
\newblock On the medianwidth of graphs.
\newblock \arxiv{1512.01104}.

\end{thebibliography}

\end{document}